\documentclass[11pt]{article}
\usepackage{float}
\usepackage{color}
\usepackage{subfigure}
\usepackage{amsmath}
\usepackage{amsthm}
\usepackage{amsfonts}
\usepackage{amssymb,latexsym}
\usepackage[all]{xy}
\usepackage{graphicx}
\usepackage[mathscr]{eucal}
\usepackage{verbatim}
\usepackage{hyperref}

\theoremstyle{plain}

    \newtheorem{thm}{Theorem}[section]
       
       \newtheorem{lem}{Lemma}[section]
       \newtheorem{defn}{Definition}[section]

\numberwithin{equation}{section}

\usepackage{graphicx}

\begin{document}
\title{Exponential attractors  for a nonlocal delayed reaction-diffusion equation on an unbounded domain}

\author{Wenjie Hu$^{1,2}$,  Tom\'{a}s
Caraballo$^{3,4}$\footnote{Corresponding author.  E-mail address: caraball@us.es (Tom\'{a}s
Caraballo).}
\\
\small  1. The MOE-LCSM, School of Mathematics and Statistics,  Hunan Normal University,\\
\small Changsha, Hunan 410081, China\\
\small  2. Journal House, Hunan Normal University, Changsha, Hunan 410081, China\\
\small $^3$Dpto. Ecuaciones Diferenciales y An\'{a}lisis Num\'{e}rico, Facultad de Matem\'{a}ticas,\\
\small  Universidad de Sevilla, c/ Tarfia s/n, 41012-Sevilla, Spain\\
\small $^4$Department of Mathematics, Wenzhou University, \\
\small Wenzhou, Zhejiang Province 325035, China
}

\date {}
\maketitle

\begin{abstract}
The main objective of this paper is to investigate exponential attractors for a nonlocal delayed reaction-diffusion equation on an unbounded domain.  We first obtain the existence of a globally attractive absorbing set for the dynamical system generated by the equation under the assumption that the nonlinear term is bounded. Then,  we construct exponential attractors of the equation directly in its natural phase space, i.e., a Banach space with explicit fractal dimension by combining  squeezing properties of the system as well as a covering lemma of finite subspace of Banach spaces. Our result generalizes the methods established in Hilbert spaces and weighted spaces, and the fractal dimension of the obtained exponential attractor does not depend on the entropy number but only depends on some inner characteristic of the studied equation.
\end{abstract}

\bigskip


\section{Introduction}
Consider the following nonlocal delayed reaction-diffusion equation on $\mathbb{R}^N$
\begin{equation}\label{1}
\left\{\begin{array}{l}\frac{\partial u}{\partial t}(x, t)= \Delta u(x, t)-\mu u(x, t)+\sigma u(x,t-\tau) +\varepsilon\int_{\mathbb{R}^N}\Gamma_{\iota}(x-y)b(u(y,t-\tau))\mathrm{d}y+g(x), t>0,\\ u_0(x,s)=\phi(x, s),-\tau \leq s \leq 0, x \in \mathbb{R}^N,\end{array}\right.
\end{equation}
where, $N$ is a positive integer, $u(\cdot,t)\in \mathbb{X}\triangleq L^2(\mathbb{R}^N)$ is the uniform bounded continuous  function on $\mathbb{R}^N$ with   norm $\|\cdot\|$, for any $t\geq0$. $\Delta$ is the Laplacian operator on $\mathbb{R}^N$, $\mu, \sigma, \varepsilon$ and $\tau$ are positive constants, $g\in \mathbb{X}$.  The kernel $\Gamma_\iota(x)$  parameterized by a positive constant $\iota$ is given by $\displaystyle \Gamma_{\iota}(x)=\frac{1}{{(4\pi \iota)}^{N/2}}e^{-\frac{|x|^2}{4\iota}}$. $u_t\in \mathcal{C}$ is defined by  $u_t(x,\theta)=u(x, t+\theta),-\tau \leq \theta \leq 0$, where $\mathcal{C}$ is the space of continuous functions from $[-\tau, 0]$ to $\mathbb{X}$ equipped with the supremum norm $\|\phi\|_{\mathcal{C}}=\sup_{\theta \in[-\tau, 0]}\|\phi(\theta)\|$ for any $\phi \in \mathcal{C}$.  The initial data $\phi\in \mathcal{C}$ and $b: \mathcal{C}\rightarrow \mathbb{X}$ is a nonlinear operator.  \eqref{1} arises largely from biological, physical and chemical processes. For instance, it can be used to describe the growth of mature population  of a two-stage species (juvenile and adult, with a fixed maturation time $\tau$) whose mature individuals and immature ones both diffuse, such as fish. So et al. \cite{SWZ} first derived \eqref{1}  in the case $\sigma=0, N=1, g\equiv 0$ from the biological point of view.

Most existing works on this field are only concerned with the existence and  qualitative properties of traveling wave  solutions, which may explain the invasion of species. See, for instance, \cite{FT,SWZ} and the references therein. Nevertheless, little attention has been paid to the global dynamics of \eqref{1} due to the non-compactness of the spatial domain and the nonlocal integral,  causing many nice results in dynamical system theory ineffective. Recently, Yi, Chen and Wu \cite{YCWT} overcame this difficulty by innovatively  introducing the compact open topology and showed the existence and global attractivity of a positive steady state by delicately constructing a priori estimate for nontrivial solutions under the compact open topology in the case $\sigma=0,g\equiv 0$, and assuming the nonlinear term $b$ is globally bounded and admits a unique two periodic fixed point. One naturally wonders what can we say about the dynamics of \eqref{1} if we do not impose the conditions on $b$ used in \cite{YCWT} and work in the  natural phase space $\mathcal{C}$ under the usual supremum norm? As \eqref{1} generates an infinite dimensional system, the first step should be to investigate existence of global attractors which can reduce the essential dynamics to a compact set. Furthermore, if the attractors posses finite dimension, then the limit dynamics can be described by a finite number of parameters. These have been done in \cite{WK} and our recent work \cite{HCEN}.

Nevertheless, as pointed out in \cite{C18} and \cite{C21}, the global attractor has several drawbacks. It may have slow  convergence rate, may be sensitive to perturbations and may fail to capture important transient behaviors. Hence, Eden, Foias, Nicolaenko and Temam proposed in \cite{EFNT94}  the concept of exponential attractor, which is a positive  invariant compact subset with finite fractal dimension and attracts all bounded subsets at an exponential rate. It is well known that if exponential attractors  exist, then they contain global attractors and hence  play significant roles in investigating asymptotic behavior of infinite dimensional nonlinear dynamical systems, especially for those with fast convergence rate. The method has been extended to study the existence of exponential attractors of functional differential equations in Banach spaces in our recent work in \cite{HCE}.

The above mentioned works are all concerned with delayed partial differential equations on bounded domains. On the other hand, when the domain is unbounded, several problems arise: the Laplace operator has a continuous spectrum and the solution semiflow does not have compact absorbing sets.  Hence Babin and Nicolaenko \cite{BN95} and Efendiev and Miranville \cite{EM99} investigated exponential attractors in weighted Hilbert spaces requiring the initial data and forcing term also belong to the corresponding spaces. Therefore, one natural question arises: how to construct an exponential attractor of \eqref{1} directly in the Banach space $\mathcal{C}$?  Efendiev, Miranville and  Zelik \cite{C17,C18} adopted the so-called smoothing property of the semigroup between two different Banach spaces to construct existence of  exponential and uniform attractors for systems in Banach spaces, which was generalized in \cite{C17,C18,C21,C20,LHSS20,ZZW}.  The fractal dimensions of the exponential attractors they established depend on the entropy numbers between two spaces which is in general quite difficult to obtain and may vary from space to space. In our recent work \cite{HCE}, we provide an alternative method for constructing exponential attractors for PFDEs on bounded domains with fractal dimensions that only depend on  some inner characteristics of the equation, such as  the spectrum of the linear part of the equation on a bounded domain and the Lipschitz constant of the nonlinear part. Here, we go one step further  and consider the equations on an unbounded domain. The unboundedness of the domain causes the Laplace operator has a continuous spectrum, $H^1(\mathbb{R}^n)$ is not compactly embedded in $L^2(\mathbb{R}^n)$ and the solution semiflow does not have compact absorbing sets  in the original topology, causing the method in \cite{HCE} ineffective. Hence, we propose a different method by decomposing the solution of \eqref{1} into a sum of three parts, among which, each part fulfils  the squeezing property.

The outline of our paper is as follows. In Section 2, we prove the existence of a globally attractive absorbing set for the equation under the assumption that the nonlinear term is bounded and introduce the squeezing property established in our recent work \cite{HCEN}.  Then, we construct exponential attractors of the equation directly in its natural phase space, i.e., the Banach space $\mathcal{C}$ in Section 3.

\section{Preliminaries}
In this section, we will prove the existence of  a globally attractive absorbing set of the infinite dynamical system generated by \eqref{1}. By the  Fourier transformation, we can see that the semigroup  $S(t): \mathbb{X}\rightarrow \mathbb{X}, t\geq 0$ generated by $\Delta-\mu I$ is defined as
 \begin{equation}\label{2.1}
\left\{\begin{array}{l}
S(0)[\phi](x)=\phi(x), \\
S(t)[\phi](x)=\frac{e^{-\mu t}}{ (4 \pi t)^{N/2}} \int_{\mathbb{R}^N} \phi(y) e^{ -\frac{|x-y|^{2}}{4 t} } dy, t \in(0, \infty),
\end{array}\right.
 \end{equation}
for $(x, \phi) \in \mathbb{R}^N \times \mathbb{X}$, which is  analytic and strongly continuous on $\mathbb{X}$.

Define $H: \mathbb{X} \rightarrow \mathbb{X}$ by
$$
H(\phi)(x)=\int_{\mathbb{R}^N} \Gamma_\iota(x-y) \phi(y) \mathrm{d} y
$$
for all  $\phi \in X$. Then,  by the expression of $\Gamma_\iota(\cdot)$, we have $\|H\|\triangleq \sup\{\frac{\|H(\phi)\|}{\|\phi\|}: \|\phi\|\neq 0\}\leq 1$.

We introduce the following results concerning the properties of semigroup $\{S(t)\}_{t\geq 0}$, which is frequently used throughout the whole paper. The details of the proof can be found in \cite{YCWT} lemmas 2.1 and 2.4.
\begin{lem}\label{lem2.1}
Let $\{S(t)\}_{t\geq 0}$  be defined by \eqref{2.1}, then we have the following results.\\
(i) $\|S(t) \phi\|\leq e^{-\mu t} \|\phi\|$ for all $\phi \in \mathbb{X}$, $t \in \mathbb{R}_{+}$.\\
(ii) $\{S(t)\}_{t\geq 0}$ is an analytic and strongly continuous semigroup on $\mathbb{X}$.\\
(iii) For all $t \in(0, \infty)$ and $(x, \phi) \in (0, \infty)\times \mathbb{X},$ there holds
$$
\begin{array}{l}
 S(t)[a](x)=a e^{-\mu t}.
\end{array}
$$
\end{lem}

In the remaining part, we always assume that the nonlinear term $f$ is globally bounded, that is, there exists a $N>0$ such that for any $\phi\in \mathcal{C}$, we have $f(\phi)\leq N$. By an argument of steps, we know that for any given $\phi \in \mathcal{C}$, \eqref{1} has a unique solution in $\mathcal{C}$ for all $t\geq 0$. By the variation of constants method, \eqref{1} is equivalent to the following integral equation with the given initial function
  \begin{equation}\label{2.2}
    \left\{
     \begin{array}{ll}
     \displaystyle u(t)=S(t)\phi(0) +\sigma \int_{0}^{t}S(t-s) u(s-\tau) \mathrm{d}s+ \int_{0}^{t}S(t-s)[H(f(u(s-\tau,\cdot)))+g]\mathrm{d}s,   t>0, \\u_{0}=\phi\in \mathcal{C}.
     \end{array}
     \right.
  \end{equation}
   Let $ u^{\phi}(t) $ be the solution of \eqref{2.2}. Define the infinite dimensional dynamical system $\Phi(t):  \mathcal{C}\rightarrow  \mathcal{C}$ by $\Phi(t)\phi= u^{\phi}_{t}$ for all $(t,\phi)\in\mathbb{R_{+}}\times \mathcal{C}$.

We first show that $\Phi$ admits an absorbing set.

\begin{thm}\label{thm2.1}
Assume that $f$ is bounded  and $\sigma e^{\mu \tau}-\mu<0$, then the dynamical system $\Phi$ admits an absorbing set $\mathcal{B}$ defined by
\begin{equation}\label{2.3}
\mathcal{B}=\left\{\phi \in \mathcal{C}: \|\phi\|_{ \mathcal{C}}\leq 2\left(\frac{M}{\mu}+\frac{M\sigma e^{\mu \tau}}{\mu(\mu-\sigma e^{\mu \tau})}\right)\right\},
\end{equation}
where $M=N+\|g\|$.
\end{thm}
\begin{proof}
It follows from \eqref{2.2}, Lemma \ref{lem2.1} and boundedness of $f$ that
\begin{equation}\label{2.4}
\begin{aligned}
\left\|u(t)\right\|\leq &\left\| S(t)\phi(0) \right\|+\left\|\sigma \int_{0}^{t}S(t-s) u(s-\tau) \mathrm{d}s\right\|+\left\|\int_{0}^{t}S(t-s)[H(f(u(s-\tau,\cdot)))+g]\mathrm{d}s\right\| \\
\leq & e^{-\mu t}\left\|\phi(0)\right\|+ \sigma \int_{0}^{t} e^{-\mu(t-s) } \|u(s-\tau)\| \mathrm{d} s+\int_{0}^{t} e^{-\mu(t-s) } [\|f(u(s-\tau,\cdot))\|+\|g\|] \mathrm{d} s\\
\leq & e^{-\mu t}\left\|\phi(0)\right\|+\sigma \int_{0}^{t} e^{-\mu(t-s) } \|u(s-\tau)\| \mathrm{d} s+\frac{M(1-e^{-\mu t})}{\mu}.
\end{aligned}
\end{equation}
Thus, for $\xi\in [-\tau, 0]$, we have
\begin{equation}\label{2.5}
\begin{aligned}
\left\|u(t+\xi)\right\|
\leq & e^{-\mu (t+\xi)}\left\|\phi(0)\right\|+\sigma \int_{0}^{t+\xi} e^{-\mu(t+\xi-s) } \|u(s-\tau)\| \mathrm{d} s+\frac{M(1-e^{-\mu (t-\xi)})}{\mu}\\
\leq & e^{\mu \tau}e^{-\mu t }\left\|\phi(0)\right\|+\sigma e^{\mu \tau}\int_{0}^{t} e^{-\mu(t-s) } \|u(s-\tau)\| \mathrm{d} s+\frac{M }{\mu}.
\end{aligned}
\end{equation}
Keeping in mind that $\left\|u_{t}\right\|_{\mathcal{C}}=\sup \left\{\|u(t+\xi)\|: \xi \in[-\tau, 0]\right\},$  we have
\begin{equation}\label{2.6}
\begin{aligned}
\left\|u_t\right\|_{\mathcal{C}}\leq & e^{\mu \tau}e^{-\mu t }\left\|\phi\right\|_{\mathcal{C}}+\sigma e^{\mu \tau}\int_{0}^{t} e^{-\mu(t-s) } \|u_s\|_{\mathcal{C}} \mathrm{d} s+\frac{M }{\mu}.
\end{aligned}
\end{equation}
Multiplying both sides of \eqref{2.6} by $e^{\mu t}$ implies
\begin{equation}\label{2.7}
\begin{aligned}
 e^{\mu t}\left\|u_t \right\|\leq & e^{\mu \tau}\left\|\phi\right\|_{\mathcal{C}}+\sigma e^{\mu \tau}\int_{0}^{t} e^{ \mu s } \|u_s\|_{\mathcal{C}} \mathrm{d} s+\frac{Me^{\mu t}}{\mu}.
\end{aligned}
\end{equation}
Applying Gr{o}nwall's inequality yields
\begin{equation}\label{2.8}
\begin{aligned}
e^{\mu t} \left\|u_t\right\|_{\mathcal{C}}\leq &(\frac{M}{\mu}e^{\mu t}+e^{\mu \tau}\left\|\phi\right\|_{\mathcal{C}})+\frac{M\beta e^{\beta t}}{\mu(\mu-\beta)}[e^{(\mu-\beta)t}-1]+
 (e^{\beta t}-1) e^{\mu \tau} \left\|\phi\right\|_{\mathcal{C}},
\end{aligned}
\end{equation}
and hence
\begin{equation}\label{2.9}
\begin{aligned}
  \left\|u_t\right\|_{\mathcal{C}}\leq & (\frac{M}{\mu}+e^{\mu (\tau-t)}\left\|\phi\right\|_{\mathcal{C}})+\frac{M\beta}{\mu(\mu-\beta)}[1-e^{-(\mu-\beta)t}]+
 e^{\mu \tau} \left\|\phi\right\|_{\mathcal{C}}e^{(\beta-\mu) t}\\
 \leq& (\frac{M}{\mu}+\frac{M\beta}{\mu(\mu-\beta)})+(e^{\mu (\tau-t)}\left\|\phi\right\|_{\mathcal{C}}+e^{\mu \tau} \left\|\phi\right\|_{\mathcal{C}}e^{(\beta-\mu) t}),
 \end{aligned}
\end{equation}
where $\beta=\sigma e^{\mu \tau}$. Since we have assumed that $\sigma e^{\mu \tau}-\mu<0$, one can see, for any bounded set $ \mathcal{D} \subset \mathcal{C}$, there exists a $T_{\mathcal{D}}>0$ such that for all $t\geq T_{\mathcal{D}}$
\begin{equation}\label{2.9}
\begin{aligned}
 e^{\mu (\tau-t)}\left\|\phi\right\|_{\mathcal{C}}+e^{\mu \tau} \left\|\phi\right\|_{\mathcal{C}}e^{(\beta-\mu) t}\leq \frac{M}{\mu}+\frac{M\beta}{\mu(\mu-\beta)},
 \end{aligned}
\end{equation}
indicating that
\begin{equation}\label{2.10}
\begin{aligned}
\left\|u_t\right\|_{\mathcal{C}} \leq & 2\left(\frac{M}{\mu}+\frac{M\beta}{\mu(\mu-\beta)}\right).
\end{aligned}
\end{equation}
That is, $\mathcal{B}$ is an absorbing set for $\Phi$.
This completes the proof.
\end{proof}

Throughout the remaining part of this paper, we always assume that $f$ is globally Lipschitz, that is, $f$ satisfies

$\mathbf{Hypothesis\  A1:}$
$\left\|f\left(\phi_1\right)-f\left(\phi_2\right)\right\| \leq L_f\left\|\phi_1-\phi_2\right\|_{\mathcal{C}}, \text { for any } \phi_1, \phi_2 \in \mathcal{C}$.\\
Let $\{\Phi(t)\}_{t\geq 0}$ be the dynamical system generated by \eqref{1}. Since the norm of the nonlocal operator $H\leq 1$, by using the same technique for establishing \cite[Theorem 4.1]{HCEN} with minor adjustment, one can show that in the case $\sigma e^{\mu \tau}-\mu<0$ and $\mathbf{Hypothesis\  A1}$ holds, then $\{\Phi(t)\}_{t\geq 0}$ admits a global attractor $\mathcal{A}\subset \mathcal{C}$,  which is a compact invariant set and attracts every bounded set in $\mathcal{C}$.

In order to introduce the squeezing property established in \cite{HCEN}, we introduce more notations. Let   $K$ be defined in Lemma 3.4 of \cite{HCEN} and denote the ball centered at 0 with radius $K$ by $\Omega_K= \{x \in \mathbb{R}^N:|x|
 <K \}$, its boundary by $\partial \Omega_K=\left\{x \in \mathbb{R}^N:|x|=K\right\}$ and its complement by $\Omega_K^C=\left\{x \in \mathbb{R}^N:|x|\geq K\right\}$ respectively.  Let $\mathbb{X}_{\Omega_K}$ be the Hilbert space $\mathbb{X}_{\Omega_K}=\{\phi \in L^2(\Omega_K): \phi(\partial \Omega_K)=0\}$, where $L^2(\Omega_K)$ is the square Lesbegue integrable functions on $\Omega_K$.  Denote by $\mathcal{C}_{\Omega_K}=C([-\tau, 0],\mathbb{X}_{\Omega_K})$ the set of all  continuous functions from $[-\tau, 0]$ to $\mathbb{X}_{\Omega_K}$ equipped with the usual supremum norm $\|\phi\|_{\mathcal{C}_{\Omega_K}} =\sup\{\| \phi(\xi)\|_{\mathbb{X}_{\Omega_K}} :\xi \in [-\tau,0]\}$ for all  $\phi\in \mathcal{C}_{\Omega_K}$.

Define $\chi_{\Omega_K}$ and $\chi_{\Omega_K^C}$ as the characteristic functions on $\Omega_K$ and  $\Omega_K^C$ respectively, that is
\begin{equation}\label{2.12a}
\chi_{\Omega_K}(x)=\left\{\begin{array}{l}0, x \in \Omega_K^C,\\ 1,x \in \Omega_K,\end{array}\right.
\end{equation}
and
\begin{equation}\label{2.13a}
\chi_{\Omega_K^C}(x)=\left\{\begin{array}{l}0, x \in \Omega_K,\\ 1, x \in \Omega_K^C.\end{array}\right.
\end{equation}
 Set $u(x,t)=v(x,t)+w(x,t)$, with $v(x,t)=u(x,t) \chi_{\Omega_K}(x)$ and $w(x,t)=u(x,t) \chi_{\Omega_K^C}(x)$
for any $t \geq -\tau$ and $x \in \mathbb{R}^N$, then $v(t, x)$ and $w(t, x)$ satisfy
\begin{equation}\label{2.14}
\left\{\begin{array}{l}
\frac{\partial v(x,t)}{\partial t}=\Delta v(x,t)-\mu v(x,t)+\sigma v(x,t-\tau)+H(f \left(u(\cdot,t-\tau)\right))(x)\chi_{\Omega_K}(x)+g(x)\chi_{\Omega_K}(x), \\
v(x,s)=\phi(x, s) \chi_{\Omega_K}(x)\triangleq \varphi, s \in[-\tau, 0], x \in \mathbb{R}^N,\\
v(x,t)=0, t \in(0, \infty), x \in \partial \Omega_K
\end{array}\right.
\end{equation}
and
\begin{equation}\label{2.4}
\left\{\begin{array}{l}
\frac{\partial w(x,t)}{\partial t}=\Delta w(x,t)-\mu w(x,t)+\sigma w(x,t-\tau)+H(f \left(u(\cdot,t-\tau)\right))(x)\chi_{\Omega_K^C}(x)+g(x)\chi_{\Omega_K^C}(x), \\
w(x,s)=\phi(x, s) \chi_{\Omega_K^C}(x)\triangleq \psi, s \in[-\tau, 0], x \in \mathbb{R}^N,
\end{array}\right.
\end{equation}
respectively, in which $u$ is the solution to \eqref{1}.

We first consider the following linear part of \eqref{2.3} on $\mathbb{X}_{\Omega_K}$.
\begin{equation}\label{5.1}
\left\{\begin{array}{l}
\frac{\partial \tilde{v}(x,t)}{\partial t}=\Delta \tilde{v}(x,t)-\mu \tilde{v}(x,t)+\sigma \tilde{v}(x,t-\tau), \\
\tilde{v}(x,s)=\phi(x, s)\chi_{\Omega_K}(x), s \in[-\tau, 0], x \in  \Omega_K,\\
\tilde{v}(x,t)=0, t \in(0, \infty), x \in \partial \Omega_K.
\end{array}\right.
\end{equation}

It follows from \cite[Theorem 2.6]{WJ}  that   \eqref{5.1} admits a global solution   $\tilde{v}^\phi(\cdot):[-r, \infty] \rightarrow \mathbb{X}_{\Omega_K}$. Define the linear semigroup  $U(t): \mathcal{C}_{\Omega_K}\rightarrow \mathcal{C}_{\Omega_K}$ by $U(t)\phi=\tilde{v}^\phi_t(\cdot)$. Let  $A_U: \mathcal{C}_{\Omega_K}\rightarrow \mathcal{C}_{\Omega_K}$ be the infinitesimal generator of $U(t)$.

Let $0<\mu_{1, K} \leq \mu_{2, K} \leq \cdots \leq \mu_{m, K} \leq \cdots, \quad \mu_{m, K} \rightarrow+\infty \quad \text { as } \quad m \rightarrow+\infty$ be eigenvalues of the following eigenvalue problem on $\Omega_K$:
 \begin{equation}\label{3.9}
\begin{aligned}
-\Delta u(x)=\mu u(x),\left.\quad u(x)\right|_{x \in \partial \Omega_K}=0, \quad x \in \Omega_K,
\end{aligned}
\end{equation}
with corresponding eigenfunctions $\left\{e_{m, K}\right\}_{m \in \mathbb{N}}$.

Since $A_U$ is compact, it follows from Theorem 1.2 (i) in \cite{WJ} that the spectrum of $A_U$  are point spectra, which we denote by $\varrho_1>\varrho_2>\cdots$ with multiplicity $n_1, n_2,\cdots$.
Moreover, it follows from  \cite{WJ} that the characteristic values  $\varrho_1>\varrho_2>\cdots$  of  $A_U$ are the roots of the following characteristic equation
 \begin{equation}\label{3.10}
\begin{aligned}
\mu^2_{m, K} -\left(\lambda+\mu-\sigma e^{-\lambda \tau}\right) =0, m=1,2, \cdots,
\end{aligned}
\end{equation}
where  $\varrho_1$ is the first eigenvalue of  $A_U$  defined as
 \begin{equation}\label{3.11}
\begin{aligned}
\varrho_1=\max \left\{\operatorname{Re} \lambda: \mu^2_{m, K} -\left(\lambda+\mu-\sigma e^{-\lambda \tau}\right) =0\right\}, n=1,2, \cdots.
\end{aligned}
\end{equation}

It follows from Theorem 1.10 on P71 in \cite{WJ} that there exists a number such that all eigenvalues of $A_U$ lie in its left, implying that there exists at most a finite number of  eigenvalues of $A_U$ which are positive.  Hence,   there exists $m\geq 1$ such that  $\varrho_m<0$, and there is a
 \begin{equation}\label{3.12}
\begin{aligned}
k_m=n_1+n_2+\cdots+n_m
\end{aligned}
\end{equation}
dimensional  subspace $\mathcal{C}_{\Omega_K}^U$ such that $\mathcal{C}_{\Omega_K}$ is decomposed by  $A_U$  as
 \begin{equation}\label{3.12a}
\begin{aligned}
\mathcal{C}_{\Omega_K}= \mathcal{C}_{\Omega_K} ^U \bigoplus  \mathcal{C}_{\Omega_K} ^S,
\end{aligned}
\end{equation}
where $\mathcal{C}_{\Omega_K} ^S$ is the complement subspace of $ \mathcal{C}_{\Omega_K} ^U $with respect to $\mathcal{C}_{\Omega_K}$. Let $P_{k_m}$ and $Q_{k_m}$ be the projection of $\mathcal{C}_{\Omega_K}$ onto $ \mathcal{C}_{\Omega_K}^U$ and $ \mathcal{C}_{\Omega_K}^S$ respectively, that is
   \begin{equation}\label{3.12b}
\begin{aligned}
\mathcal{C}_{\Omega_K}^U=P_{k_m}\mathcal{C}_{\Omega_K}
\end{aligned}
\end{equation}
and
 \begin{equation}\label{3.12c}
\begin{aligned}
\mathcal{C}_{\Omega_K}^S=(I-P_{k_m})\mathcal{C}_{\Omega_K}=Q_{k_m}\mathcal{C}_{\Omega_K}.
\end{aligned}
\end{equation}
It follows from the definition of $P_{k_m}$ and $Q_{k_m}$ that
\begin{equation}\label{3.13}
\begin{aligned}
\left\|U(t)Q_{k_m} x\right\| & \leq K_m e^{\varrho_m t}\|x\|, & & t \geq0,
\end{aligned}
\end{equation}
where $K_m$ is a positive constant.

For any $\phi\in \mathcal{C}$, define the map $P_{\chi(\Omega_K^C)}: \mathcal{C}\rightarrow \mathcal{C} $ and $P_{\chi(\Omega_K)}: \mathcal{C}\rightarrow \mathcal{C}_{\Omega_K}$ by
\begin{equation}\label{6.1a}
\begin{aligned}
P_{\chi(\Omega_K)} \phi=\chi(\Omega_K)\phi
\end{aligned}
\end{equation}
and
\begin{equation}\label{6.1}
\begin{aligned}
P_{\chi(\Omega_K^C)} \phi=\chi(\Omega_K^C)\phi
\end{aligned}
\end{equation}
respectively, where $\chi(\Omega_K)$ and $\chi(\Omega_K^C)$ are the characteristic functions defined in \eqref{2.12a} and \eqref{2.13a}.
Then, by \eqref{3.12a}, we have
\begin{equation}\label{6.2}
\begin{aligned}
\mathcal{C}=P_{k_m}P_{\chi(\Omega_K)}\mathcal{C}  \bigoplus  Q_{k_m}P_{\chi(\Omega_K)}\mathcal{C} \bigoplus P_{\chi(\Omega_K^C)}\mathcal{C} \triangleq  \mathcal{C}_{\Omega_K}^U \bigoplus  \mathcal{C}_{\Omega_K}^S \bigoplus \mathcal{C}_{\Omega_K^C}.
\end{aligned}
\end{equation}
where $P_{k_m}$ and $Q_{k_m}$ are defined by \eqref{3.12b} and \eqref{3.12c}.
Define
\begin{equation}\label{6.2a}
\begin{aligned}
\mathcal{P}=P_{k_m}P_{\chi(\Omega_K)}, \mathcal{Q}=Q_{k_m}P_{\chi(\Omega_K)},\mathcal{R}= P_{\chi(\Omega_K^C)}.
\end{aligned}
\end{equation}
It follows from  \cite[Lemmas 5.1 and 5.2]{HCEN}  that
\begin{equation}\label{6.3}
\begin{aligned}
\|\mathcal{P}[\Phi(t)\varphi -\Phi(t)\psi ]\|_{\mathcal{C}(\Omega_K)} \leq e^{(L_f+\varrho_1)t}\|\varphi-\psi\|_{\mathcal{C}},
\end{aligned}
\end{equation}
 \begin{equation}\label{6.4}
\begin{aligned}
\|\mathcal{Q}[\Phi(t)\varphi -\Phi(t)\psi ]\|_{\mathcal{C}(\Omega_K)} \leq (K_me^{\varrho_m t}+\frac{K_mL_f }{\varrho_1+L_f-\varrho_m} e^{(L_f+\varrho_1)t})\|\varphi-\psi\|_{\mathcal{C}}
\end{aligned}
\end{equation}
and
\begin{equation}\label{6.5}
\begin{aligned}
\|\mathcal{R}[\Phi(t)\varphi -\Phi(t)\psi]\|_{\mathcal{C}(\Omega_K^C)} \leq \sqrt{c_2}e^{\frac{1}{2}[c_2(\sigma+ L_{f}^{2})-(\mu-\sigma-1)]t}\|\varphi-\psi\|_{\mathcal{C}}.
\end{aligned}
\end{equation}
Moreover, by \eqref{3.12} and \eqref{3.12a}, we can see $\mathcal{P}$ has $k_m$ dimension range space, that is $P_{k_m}\mathcal{C}_{\Omega_K}$ is a  $k_m$ dimensional subspace of $\mathcal{C}$.

\section{Exponential attractors}
This section is devoted to the construction of exponential attractors based on the squeezing properties \eqref{6.3}-\eqref{6.5} and the global attractor $\mathcal{A}$. We first   give the following definition of exponential attractors.
\begin{defn}\label{defn2.1} A  non-empty compact subset $\mathcal{M}\subset \mathcal{C}$ is called an
exponential attractor of the dynamical system $\{\Phi(t)\}_{t\geq 0}$ if\\
(i) $\mathcal{M}$ is positively invariant, i.e.
$$
\Phi(t)\mathcal{M} \subset \mathcal{M}  \quad \forall t \in \mathbb{R}^+.
$$\\
(ii) The fractal dimension $\operatorname{dim}_f(\mathcal{M})$ of $\mathcal{M}$ is bounded,
where $\operatorname{dim}_f(\mathcal{M})$ is  defined as
$$
\operatorname{dim}_f(\mathcal{M})=\lim _{\varepsilon \rightarrow 0} \frac{\ln \left(N_{\varepsilon}^X(\mathcal{M})\right)}{\ln \left(\frac{1}{\varepsilon}\right)},
$$
and $N_{\varepsilon}^X(\mathcal{M})$ denotes the minimal number of $\varepsilon$-balls in $X$ with centres in $X$ needed to cover $\mathcal{M}$.\\
(iii) $\mathcal{M}$ exponentially attracts all bounded sets, i.e. there exists a constant $\omega>0$ such that for every bounded subset $D \subset X$
$$
\lim _{t \rightarrow \infty} e^{-\omega t} \operatorname{dist}_X(\Phi(t) D, \mathcal{M})=0,
$$
where $\mathrm{dist}_X(A,B)$ denotes the Hausdorff semi-distance  between $A$ and $B$, defined as
$$
\mathrm{dist}_X(A, B)=\sup _{a \in A} \inf _{b \in B} d(a, b), \quad \text { for } A, B \subseteq X.
$$
\end{defn}

We first construct exponential  attractors for the discrete version $\{\Phi(n) \}$ of $\{\Phi(t)\}_{t\geq 0}$.
\begin{thm}\label{thm2.2} Let $k_m, \varrho_{1}, \varrho_{m}$ and $K_m$ be  defined in \eqref{3.11}  and \eqref{3.13} respectively,  $\mathcal{P}, \mathcal{Q}$ and $\mathcal{R}$ be  defined by \eqref{6.2a}, and $\mathcal{A}$ is the global attractor. Assume that $f$ is bounded, $\sigma e^{\mu \tau}-\mu<0$ and $\mathbf{Hypothesis\  A1}$ holds.  Moreover, assume  that there exists $\alpha>0$ such that $\zeta:=\alpha e^{(L_f+\varrho_1)}+K_me^{\varrho_m}+\frac{K_mL_f }{\varrho_1+L_f-\varrho_m}e^{(L_f+\varrho_1)}+\sqrt{c_2}e^{\frac{[c_2(\sigma+ L_{f}^{2})-(\mu-\sigma-1)]}{2}}<1$. Then, there exists an exponential attractor $\mathcal{M}$ for the discrete system $ \{\Phi(n)\}$, and the fractal dimension  is bounded  by
 \begin{equation}\label{4.5}
\begin{aligned}
\operatorname{dim}_f \mathcal{M}\leq \frac{\ln k_m +k_m \ln(2+\frac{2}{\alpha})}{-\ln \zeta}<\infty.
\end{aligned}
\end{equation}
\end{thm}
\begin{proof}
\textbf{1) Covering of} $\Phi(m)\mathcal{A}$. The first step is to  inductively  construct a family of sets  $W^m, m\in \mathbb{N}$ that  satisfy the following properties
 \begin{equation}\label{4.6}
\left\{\begin{array}{l}
(W1) \quad W^m \subset \Phi(m) \mathcal{A}\subset \mathcal{A},\\
(W2) \quad \sharp W^m \leq [\Lambda 2^\Lambda \left(1+\frac{1}{\alpha} \right)^\Lambda]^m,\\
(W3) \quad \Phi(m) \mathcal{A}\subset \bigcup_{u \in W^m } B_{\zeta^m R_{\mathcal{A}}}(u),
\end{array}\right.
\end{equation}
and covers $\Phi(m)\mathcal{A}$, where $\sharp W^m$ represents the number of elements of $W^m$.

We first consider   construct a covering of   $\Phi(1)\mathcal{A}$. Since $\mathcal{A}$ is bounded,  there exists a constant $R_{\mathcal{A}}$ such that  $R_{\mathcal{A}}:=\sup _{u \in \mathcal{A}}\|u\|_X$. Then for any $u_{1}\in \mathcal{A}$, we have
$\mathcal{A}\subset B_{R_{\mathcal{A}}}\left(u_{1}\right)$. For any $u \in \mathcal{A}\cap B\left(u_{1}, R_{\mathcal{A}}\right)$, it follows from the squeezing properties \eqref{6.3}-\eqref{6.5} that
\begin{equation}\label{2.8}
\begin{aligned}
\|\mathcal{P}[\Phi(1)u-\Phi(1)u_1 ]\|_{\mathcal{C}(\Omega_K)} \leq e^{(L_f+\varrho_1)}\|u-u_1\|_{\mathcal{C}},
\end{aligned}
\end{equation}
 \begin{equation}\label{2.9}
\begin{aligned}
\|\mathcal{Q}[\Phi(1)u-\Phi(1)u_1 ]\|_{\mathcal{C}(\Omega_K)} \leq (K_me^{\varrho_m}+\frac{K_mL_f }{\varrho_1+L_f-\varrho_m} e^{(L_f+\varrho_1)})\|u-u_1\|_{\mathcal{C}}
\end{aligned}
\end{equation}
and
\begin{equation}\label{2.9a}
\begin{aligned}
\|\mathcal{R}[\Phi(1)u-\Phi(1)u_1]\|_{\mathcal{C}(\Omega_K^C)} \leq \sqrt{c_2}e^{\frac{1}{2}[c_2(\sigma+ L_{f}^{2})-(\mu-\sigma-1)]}\|u-u_1\|_{\mathcal{C}}.
\end{aligned}
\end{equation}
By Lemma  2.1 in \cite{30}, we can find $y_{1}^1, \ldots, y_{1}^{n_1}\in P\Phi(1)\mathcal{A}$ and $0<\alpha<1$ such that
 \begin{equation}\label{2.10}
\begin{gathered}
B_{\mathcal{P} \mathcal{C}}\left(\mathcal{P} \Phi(1) u_1, e^{\lambda_0}R_{\mathcal{A}}\right) \subset \bigcup_{j=1}^{n_1} B_{\mathcal{P} \mathcal{C}}\left(y_{1}^j, \alpha e^{\lambda_0}R_{\mathcal{A}}\right)
\end{gathered}
\end{equation}
with
 \begin{equation}\label{2.11}
\begin{gathered}
n_1\leq \Lambda 2^\Lambda \left(1+\frac{1}{\alpha}\right)^\Lambda,
\end{gathered}
\end{equation}
where $\Lambda$ is the dimension of $\mathcal{P C}$ and we have denoted by $B_{\mathcal{P} \mathcal{C}}(y, r)$ the ball in $\mathcal{P C}$ of radius $r$ and center $y$.
Take $t_0=1$ and set
 \begin{equation}\label{2.12}
\begin{gathered}
u_{1}^j=y_{1}^j+\mathcal{Q}\Phi(1) u_{1}+\mathcal{R}\Phi(1) u_{1}
\end{gathered}
\end{equation}
for $ j=1, \ldots, n_1$ and $W^1=\{u_{1}^1, u_{1}^2,\cdots, u_{1}^{n_1}\}$. Then, for any $u \in \mathcal{A}\cap B\left(u_{1}, R_{\mathcal{A}}\right)$, there exists a $j$ such that
 \begin{equation}\label{2.13}
\begin{aligned}
&\left\|\Phi\left(1\right) u-u_1^j\right\|\\
&\quad \leq\left\|\mathcal{P}\Phi\left(1\right)  u-y_1^j\right\|+\left\|\mathcal{Q } \Phi\left(1\right)  u-\mathcal{Q } \Phi\left(1\right) u_1\right\|+ \left\|\mathcal{R} \Phi\left(1\right)  u-\mathcal{R} \Phi\left(1\right) u_1\right\|\\
&\quad \leq\left(\beta e^{(L_f+\varrho_1)}+K_me^{\varrho_m}+\frac{K_mL_f }{\varrho_1+L_f-\varrho_m}e^{(L_f+\varrho_1)}+\sqrt{c_2}e^{\frac{[c_2(\sigma+ L_{f}^{2})-(\mu-\sigma-1)]}{2}}\right)R_\mathcal{A},
\end{aligned}
\end{equation}
indicating $(W3)$ is satisfied for $m=1$. Furthermore, it is clear from the definition of $W^1$ that it satisfies $(W1)$ and $(W2)$.
This completes the proof of case $m=1$.

Assume that the sets $W^l$ satisfying \eqref{4.6} have been already constructed for all $m\leq l$, i.e.,
there exists covering
 \begin{equation}\label{2.14}
\Phi(l) \mathcal{A}\subset \bigcup_{u \in W^l} B_{\zeta^{l}R_{\mathcal{A}}}(u).
\end{equation}
We construct  in the sequel the covering of $W^{l+1}$ that satisfies \eqref{4.6}.  By the semigroup property, we have
 \begin{equation}\label{2.15}
\begin{aligned}
\Phi(l+1) \mathcal{A}& =\Phi(1) \Phi(l) \mathcal{A} \subset \bigcup_{u \in W^l} \Phi(1) B_{\zeta^lR_{\mathcal{A}}}(u),
\end{aligned}
\end{equation}
that is $\Phi(l+1) \mathcal{A}$ can be covered by $\bigcup_{u \in W^l} \Phi(1) B_{\zeta^lR_{\mathcal{A}}}(u)$. We construct in the following a covering of $\bigcup_{u \in W^l} \Phi(1) B_{\zeta^lR_{\mathcal{A}}}(u)$.

Let $u_l \in W^l$. It follows from the induction hypotheses  $(W1)$ and $(W3)$ that
 \begin{equation}\label{2.16}
u_l \in \Phi(l) \mathcal{A}\subset \bigcup_{u \in W^l} B_{\zeta^{l}R_{\mathcal{A}}}(u).
\end{equation}
 Therefore, for any $u_l \in \mathcal{A}\cap \bigcup_{u \in W^l} \Phi(1) B_{\zeta^lR_{\mathcal{A}}}(u_l)$, using again the squeezing properties,
 \begin{equation}\label{2.17}
\begin{aligned}
\|\mathcal{P}[\Phi(1)u-\Phi(1)u_l]\|_{\mathcal{C}(\Omega_K)} \leq e^{(L_f+\varrho_1)}\zeta^lR_{\mathcal{A}},
\end{aligned}
\end{equation}
 \begin{equation}\label{2.18}
\begin{aligned}
\|\mathcal{Q}[\Phi(1)u-\Phi(1)u_l]\|_{\mathcal{C}(\Omega_K)} \leq (K_me^{\varrho_m t}+\frac{K_mL_f }{\varrho_1+L_f-\varrho_m} e^{(L_f+\varrho_1)})\zeta^lR_{\mathcal{A}}
\end{aligned}
\end{equation}
and
\begin{equation}\label{2.18a}
\begin{aligned}
\|\mathcal{R}[\Phi(1)u-\Phi(1)u_l]\|_{\mathcal{C}(\Omega_K^C)} \leq \sqrt{c_2}e^{\frac{1}{2}[c_2(\sigma+ L_{f}^{2})-(\mu-\sigma-1)]}\zeta^lR_{\mathcal{A}}.
\end{aligned}
\end{equation}
By Lemma  2.1 in \cite{30}, we can find $y_l^1, \ldots, y_l^{n_l}\in \Phi(1) \Phi(l) \mathcal{A}$ such that
 \begin{equation}\label{2.19}
\begin{gathered}
B_{P X}\left(P \Phi(1) u_l,  M_1e^{\lambda_0 }\zeta^lR_{\mathcal{A}}\right) \subset \bigcup_{j=1}^{n_l} B_{\mathcal{P} \mathcal{C}}\left(y_l^j, \alpha e^{\lambda_0}\zeta^lR_{\mathcal{A}}\right)
\end{gathered}
\end{equation}
with
 \begin{equation}\label{2.20}
\begin{gathered}
n_l\leq \Lambda 2^\Lambda \left(1+\frac{1}{\alpha}\right)^\Lambda,
\end{gathered}
\end{equation}
where $\Lambda$ is the dimension of $P X$.
Set
 \begin{equation}\label{2.21}
\begin{gathered}
u_l^j=y_l^j+(I-P) \Phi(1) u_l
\end{gathered}
\end{equation}
for $ j=1, \ldots, n_l$. Then, for any $u \in \mathcal{A}\cap \bigcup_{u \in W^l} \Phi(1) B_{\zeta^lR_{\mathcal{A}}}(u)$, there exists a $j$ such that
 \begin{equation}\label{2.22}
\begin{aligned}
\left\|\Phi(1) u-u_l^j\right\|
& \leq\left\|\mathcal{P} \Phi(1) u-y_l^j\right\|+\left\|\mathcal{Q}\Phi(1) u-\mathcal{Q} \Phi(1) u_l\right\|+\left\|\mathcal{R}\Phi(1) u-\mathcal{R} \Phi(1) u_l\right\| \\
& \leq\left(\alpha e^{(L_f+\varrho_1)}+K_me^{\varrho_m}+\frac{K_mL_f }{\varrho_1+L_f-\varrho_m}e^{(L_f+\varrho_1)}+\sqrt{c_2}e^{\frac{[c_2(\sigma+ L_{f}^{2})-(\mu-\sigma-1)]}{2}}\right)\zeta^lR_{\mathcal{A}}\\
&=\zeta^{l+1}R_{\mathcal{A}}.
\end{aligned}
\end{equation}
This implies that $ \bigcup_{u \in W^l} \Phi(1) B_{\zeta^lR_{\mathcal{A}}}(u)$ is  covered by balls with radius $\zeta^{l+1}R_{\mathcal{A}}$ and centers $\{u_l^1, u_l^2, \\
\cdots, u_l^{n_l}\}$ and hence $(W3)$ holds. Denote the new set of centres by $W^{l+1}$. From the  induction hypothesis, we have $\sharp W^{l}\leq [\Lambda 2^\Lambda \left(1+\frac{1}{\alpha} \right)^\Lambda]^l$, which yields $\sharp W^{l+1} \leq n_l \sharp W^{l} \leq [\Lambda 2^\Lambda \left(1+\frac{1}{\alpha} \right)^\Lambda]^{l+1}$ and proves $(W2)$. By construction the set of centres $W^{l+1}$, we can see $W^{l+1} \subset \Phi(1) \Phi(l) \mathcal{A}=\Phi(l+1)\mathcal{A}$,  which concludes the proof of the properties $(W1)$.

\textbf{2) Construction of random exponential attractor for $\{\Phi(n) \}$.}
We define $E^1:=W^1$  and set
 \begin{equation}\label{2.24}
\begin{aligned}
E^{n+1}:=W^{n+1}\cup \Phi(1) E^{n}, \quad n \in \mathbb{N}^+.
\end{aligned}
\end{equation}
Then, if follows from the definition of the sets $E^n(k)$, the properties of the sets $W^n$ and the positive invariance of the absorbing set $\mathcal{A}$ that the family of sets $E^n, n \in \mathbb{Z}^+$ satisfies\\
(E1) $\quad \Phi(1) E^{n-1} \subset E^{n}, \quad E^n \subset \Phi(n)\mathcal{A}$,\\
(E2) $\quad E^n=\bigcup_{i=0}^n \Phi(l) W^{n-i}, \quad \sharp E^n \leq \sum_{i=0}^n (\Lambda 2^\Lambda \left(1+ \alpha \right)^\Lambda)^i$,\\
(E3) $\quad \Phi(n) \mathcal{A}\subset \bigcup_{u \in E^n} B_{\zeta^n R_{\mathcal{A}}}(u)$.

Based on the family of sets $E^n$, we define $\mathcal{M}:=\overline{\bigcup_{n \in \mathbb{N}^+} E^n}$ and show that it yields an exponential  attractor for the semigroup $\{\Phi(n)\}$.

\textbf{Positive invariance of  $\mathcal{M}$.} It follows from property $(E 1)$ that, for all $l \in  \mathbb{N}^+$, we have
  \begin{equation}\label{2.25}
\begin{aligned}
\Phi(l) \bigcup_{n \in \mathbb{N}^+} E^n &=\bigcup_{n \in \mathbb{N}^+} \Phi(l) E^n\subset \bigcup_{n \in \mathbb{N}^+} E^{n+l}\subset \bigcup_{n \in \mathbb{N}^+} E^n.
\end{aligned}
\end{equation}
Thanks to the continuous property in $\left(\mathcal{H}_1\right)$, we can take closure in both sides of \eqref{2.25}, giving rise to
 \begin{equation}\label{2.26}
\begin{aligned}
\Phi(l) \mathcal{M}& :=\Phi(l) \overline{\bigcup_{n \in \mathbb{N}^+} E^n} &=\overline{\bigcup_{n \in \mathbb{N}^+} \Phi(l) E^n}\subset \overline{\bigcup_{n \in \mathbb{N}^+} E^{n+l}}\subset \overline{\bigcup_{n \in \mathbb{N}^+} E^n}=\mathcal{M}.
\end{aligned}
\end{equation}

\textbf{Compactness and finite dimensionality of $\mathcal{M}$.} We  prove in the sequel that  the set $\mathcal{M}$ is non-empty, precompact and of finite fractal dimension.  It follows from $(E1)$ that  for any $l \in \mathbb{N}^+$, it holds
 \begin{equation}\label{2.27}
\begin{aligned}
E^l & \subset \Phi(l) \mathcal{A}.
\end{aligned}
\end{equation}
Thus, for any $l \in \mathbb{N}^+$ we have
 \begin{equation}\label{2.28}
\begin{aligned}
\bigcup_{n \in \mathbb{N}^+} E^n=\bigcup_{n=0}^l E^n \cup \bigcup_{n=l+1}^{\infty} E^n \subset \bigcup_{n=0}^l E^n \cup \Phi(l) \mathcal{A}.
\end{aligned}
\end{equation}
Due to $\zeta<1$, for any  $\varepsilon>0$, there exists $l \in \mathbb{N}$ such that
 \begin{equation}\label{2.29}
\begin{aligned}
\zeta^{l+1} R_{\mathcal{A}} \leq \varepsilon<\zeta^{l}  R_{\mathcal{A}},
\end{aligned}
\end{equation}
which combined with the fact
 \begin{equation}\label{2.30}
\begin{aligned}
\Phi(l) \mathcal{A}\subset \bigcup_{u \in W^l} B_{\varepsilon}(u),
\end{aligned}
\end{equation}
ensures that the estimate of the number of $\varepsilon$-balls in $X$ needed to cover $\bigcup_{n\in \mathbb{N}^+} E^n$ is
 \begin{equation}\label{2.31}
\begin{aligned}
N_{\varepsilon}\left(\bigcup_{n\in \mathbb{N}^+} E^n\right)& \leq \sharp\left(\bigcup_{n=0}^l E^l\right)+\sharp W^l \leq(l+1) \sharp E^l+\left[\Lambda 2^\Lambda \left(1+\frac{1}{\alpha} \right)^\Lambda\right]^l \\
& \leq2(l+1)^2 \left[\Lambda 2^\Lambda \left(1+\frac{1}{\alpha} \right)^\Lambda\right]^l.
\end{aligned}
\end{equation}
 This proves the precompactness of $\bigcup_{n\in \mathbb{N}^+} E^n$ in $X$, which directly implies the closure $\mathcal{M}:=\overline{\bigcup_{n\in \mathbb{N}^+} E^n}$ is compact in $X$, since  $X$ is a Banach space.

It follows from \eqref{2.29} and \eqref{2.31} that   the fractal dimension of the set  $\mathcal{M}$ can be estimated by
 \begin{equation}\label{2.32}
\begin{aligned}
\operatorname{dim}_f \mathcal{M} & =\limsup_{\varepsilon \rightarrow 0} \frac{\ln N_{\varepsilon}(\mathcal{M})}{-\ln \varepsilon}\\
& \leq \limsup_{l \rightarrow \infty} \frac{\ln 2(l+1)^2+ \ln [\Lambda 2^\Lambda \left(1+ \frac{1}{\alpha} \right)^\Lambda]^l}{-\ln (\zeta^l R_{\mathcal{A}})}\\
&=\frac{\ln\Lambda +\Lambda\ln(2+ \frac{2}{\alpha})}{-\ln \zeta}<\infty.
\end{aligned}
\end{equation}

\textbf{3) Exponential attraction of $\mathcal{M}$.} We need to show that the set $\mathcal{M}$ exponentially attracts all bounded subsets of $X$ at time $l \in \mathbb{N}^+$. It follows from assumptions $\left(\mathcal{H}_1\right)$ that, for any bounded subset $D \subset X$, there exists  $n_{D} \in \mathbb{N}^+$ such that $\Phi(l) D \subset \mathcal{A}$ for all $l \geq n_{D}$.
If $l \geq n_{D}+1$, that is $l\geq n_{D}+n_0$ with some $n_0 \in \mathbb{N}$, then
 \begin{equation}\label{2.33}
\begin{aligned}
\operatorname{dist}_{X} \left(\Phi(l) D, \mathcal{M}\right)&=\operatorname{dist}_{X}\left(\Phi(l) D, \overline{\bigcup_{n=0}^{\infty}E^n}\right)  \\& \leq \operatorname{dist}_{X}\left(S\left(n_0\right) S\left(l-n_0\right) D, \bigcup_{n=0}^{\infty} E^n\right) \\
& \leq \operatorname{dist}_{X}\left(S\left(n_0\right) \mathcal{A}, \bigcup_{n=0}^{\infty} E^n\right) \\
& \leq \operatorname{dist}_{X}\left(S\left(n_0\right) \mathcal{A}, E^{n_0}\right) \\
& \leq \zeta^{n_0} R_{ \mathcal{A}} \leq c e^{-\omega n}
\end{aligned}
\end{equation}
for some constants $c \geq 0$ and $\omega>0$, since $\zeta<1$. This completes the proof.
\end{proof}

By adopting the same procedure as the proof of  \cite[Theorem 3.2]{C9}, we have  the following  results about the existence of exponential attractors for continuous semigroup in Banach spaces.

\begin{thm}\label{thm2.3} Assume the conditions of Theorem \ref{thm2.2} hold. Then, the dynamical system $\{ \Phi(t) \}_{t\geq0}$ admits an exponential attractor $\mathcal{M}$ whose fractal dimension  is bounded  by
 \begin{equation}\label{2.34}
\begin{aligned}
\operatorname{dim}_f \mathcal{A}\leq \frac{\ln\Lambda +\Lambda\ln(2+ \frac{2}{\alpha})}{-\ln \zeta}<\infty,
\end{aligned}
\end{equation}
where $\zeta$ is defined in Theorem \ref{thm2.1}.
\end{thm}

\noindent{\bf Acknowledgement.}
This work was jointly supported by the Scientific Research Fund of Hunan Provincial Education Department (23C0013), China Scholarship Council(202008430247). \\
The research of T. Caraballo has been partially supported by Spanish Ministerio de Ciencia e
Innovaci\'{o}n (MCI), Agencia Estatal de Investigaci\'{o}n (AEI), Fondo Europeo de
Desarrollo Regional (FEDER) under the project PID2021-122991NB-C21.\\
This work began when Wenjie Hu was visiting the Universidad de Sevilla as a visiting scholar, and he would like to thank the staff in the Facultad de Matem\'{a}ticas  for their hospitality and thank the university for its excellent facilities and support during his stay.

\end{document}